\documentclass[12pt]{amsart}
\usepackage[english]{babel}
\usepackage[T1]{fontenc}
\usepackage{fullpage}

\def \Re{\mbox{${\mathcal Re}$}}
\def \r{\mbox{${\mathbb R}$}}
\def \C{\mbox{${\mathbb C}$}}
\def \h{\mbox{${\mathbb H}$}}
\def \E {\mathbb{E}}
\def \L {\mathbb{L}}
\def \K {\mathbb{K}}

\newtheorem{thm}{Theorem}[section]

\newtheorem{lem}[thm]{Lemma}

\theoremstyle{remark}
\newtheorem{rem}[thm]{Remark}

\theoremstyle{definition}
\newtheorem{dfn}[thm]{Definition}
\newtheorem{ex}[thm]{Example}

\numberwithin{equation}{section} \numberwithin{thm}{section}

\begin{document}



\title{The Bj\"{o}rling problem for minimal surfaces in a Lorentzian three-dimensional Lie group}
\author{Adriana A. Cintra}
\address{Departamento de Matem\'{a}tica, C.P. 6065\\
IMECC, UNICAMP, 13081-970, Campinas, SP\\ Brasil}

\email{ra099847@ime.unicamp.br}

\author{Francesco Mercuri}
\address{Departamento de Matem\'{a}tica, C.P. 6065\\
IMECC, UNICAMP, 13081-970, Campinas, SP\\ Brasil}

\email{mercuri@ime.unicamp.br}

\author{Irene I. Onnis}

\address{Departamento de Matem\'{a}tica, C.P. 668\\ ICMC,
USP, 13560-970, S\~{a}o Carlos, SP\\ Brasil}

\email{onnis@icmc.usp.br}

\thanks{Work partially supported by Capes and CNPq}

\begin{abstract}
In this paper we will show the existence and uniqueness of the solution of the Bj\"{o}rling problem for minimal surfaces in a 3-dimensional Lorentzian Lie group.
\end{abstract}

\maketitle

\section{Introduction}

The  Weierstrass representation formula for minimal surfaces in
$\r^3$ has been a fundamental tool for producing examples and
proving general properties of such surfaces, since the surfaces can be
parametrized by holomorphic data. In \cite{MMP} the authors describe
a general Weierstrass representation formula for simply connected minimal surfaces in
an arbitrary Riemannian manifold. The partial differential equations
involved are, in general, too complicated to be solved explicitly.
However, for particular ambient 3-manifolds, such as the Heisenberg
group, the hyperbolic space and the product of the hyperbolic plane
with $\r$, the equations are more manageable and the formula can be
used to produce examples (see \cite{K}, \cite{MMP}).

In the Lorentz-Minkowski space $\L^3$, i.e. the affine three space $\mathbb{R}^3$ endowed with the Lorentzian metric
$$ds^2=dx_1^2+dx_2^2-dx_3^2,$$ a
Weierstrass representation type theorem was proved by Kobayashi \cite{Kob} for  spacelike minimal immersions
and by Konderak \cite{konderak} for the case of timelike minimal surfaces.
Recently, these theorems were extended for minimal surfaces in Riemannian and Lorentzian 3-dimensional manifolds
 by Lira, Melo and Mercuri (see \cite{Liramm}).

The aim of this paper is show how the Weierstrass representation formula can be used, if
the ambient manifold is a $3$-dimensional Lorentzian Lie group, in order to
prove existence and uniqueness of the solution of the Bj\"{o}rling
problem. We remember that the classical Bj\"{o}rling problem, proposed by Bj\"{o}rling (see \cite{B}) in 1844, asks for the construction of a minimal surface in $\r^3$ containing a given analytic curve $\beta$ with a given analytic unit normal $V$ along it. The problem was solved by H.A. Schwarz (see \cite{S}) in 1890 by means of an integral  formula in terms of $\beta$ and $V$. Some extensions of this problem in others ambient spaces have been proposed and solved in \cite{ACM, AV, CDM, MO}.

The paper is organized as follows. In Section~\ref{algL} we recall some basics facts of Lorentzian calculus, which play the role of complex calculus in the classical case, for timelike minimal surfaces. Section~\ref{weier} is devoted to present a Weierstrass type representation for minimal surfaces in Lorentzian 3-dimensional manifolds, following \cite{Liramm}. In Section~\ref{bjorling} we state and solve the Bj\"{o}rling problem for timelike and spacelike minimal surfaces in a Lorentzian 3-dimensional Lie group. For timelike minimal surfaces this is done by consideration  two different cases: when $\beta$ is a timelike curve we will call the corresponding problem the {\it timelike Bj\"{o}rling problem}, and when $\beta$ is a spacelike curve we will have the {\it spacelike Bj\"{o}rling problem}.

In Sections~\ref{lorentzianH}, \ref{sitter} and \ref{H2xR} we present some examples of minimal surfaces constructed via Bj\"{o}rling problem for the case in which the ambient manifold is the Heisenberg group $\h_3$, the de Sitter space $\mathbb{S}^3_1$ and the space $\mathbb{H}^2\times\mathbb{R}$, equipped with  left-invariant Lorentzian metrics.

\section{The algebra $\L$ of paracomplex numbers}\label{algL}

In \cite{konderak}, the author use paracomplex analysis to deduce a Weierstrass representation formula for timelike minimal surfaces in $\L^3$.

We recall that the algebra of paracomplex (or Lorentz) numbers is the algebra $$\L = \{a + \tau\, b\;|\; a,b \in \r\},$$ where $\tau$ is an imaginary
unit with $\tau^2 = 1$. The two internal operations are the obvious ones. We define the conjugation in $\L$ as $\overline{a + \tau\, b} := a - \tau\, b$ and the
$\L$-norm of $z = a + \tau b \in \L$ is defined by $$||z|| = |z\,\bar{z}|^{\frac{1}{2}} = |a^2 - b^2|^{\frac{1}{2}}.$$
The algebra $\L$ admits the set consisting of zero divisors $K = \{a\pm \tau\,a \in \L \;:\; a\neq 0\}$.
If $z\notin K\cup\{0\}$, then $z$ is invertible with inverse $\displaystyle z^{-1} = \bar{z}/(z\bar{z})$.

We have that $\L$ is isomorphic to the algebra $\r \oplus \r$ via the map:
$$
\rho(a + \tau\, b) = (a + b, a - b).
$$

The set $\L$ has a natural topology as a $2$-dimensional real vector space.
\begin{dfn}
Let $\Omega\subseteq \L$ be an open set and $z_0 \in \Omega$.
The $\L$-derivative of a function $f:\Omega\rightarrow\L$ at $z_0$ is defined by
\begin{equation}\label{eq:A11}
f'(z_0):= \lim_{z\rightarrow z_{0} \atop{z - z_{0} \in \L\setminus K\cup\{0\}}}\frac{f(z) - f(z_0)}{z - z_0},
\end{equation}
if the limit exists. If $f'(z_0)$ exists, we will say that $f$ is $\L$-differentiable
at $z_0$.
\end{dfn}

\begin{rem}
The condition of $\L$-differentiability is much less restrictive that the usual complex differentiability.
For example, $\L$-differentiability at $z_0$ does not imply continuity at $z_0$. However, $\L$-differentiability
in an open set $\Omega \subset \L$ implies usual differentiability in $\Omega$.
\end{rem}

Introducing the paracomplex operators:
\begin{equation}\label{dpc}
\frac{\partial}{\partial z} = \frac{1}{2}(\frac{\partial}{\partial u} + \tau\frac{\partial}{\partial v}),\qquad
\frac{\partial}{\partial \bar{z}} = \frac{1}{2}(\frac{\partial}{\partial u} - \tau\frac{\partial}{\partial v}),\nonumber
\end{equation}
where $z = u + \tau\, v$, we have that a differentiable function $f:\Omega\rightarrow\L$ is $\L$-differentiable if and only if
\begin{equation}\label{eqldif}
\displaystyle\frac{\partial f}{\partial \bar{z}} = 0.
\end{equation}
We observe that, writing  $f(u,v) = a(u,v) +\tau\, b(u,v)$, $u + \tau v\in\Omega$, the condition~\eqref{eqldif} is equivalent to the para-Cauchy-Riemann equations:
\begin{equation} \left\{\begin{aligned}\label{eq:A14}
\frac{\partial a}{\partial u} &= \frac{\partial b}{\partial v},\\
\frac{\partial a}{\partial v} &= \frac{\partial b}{\partial u},
\end{aligned}\right.
\end{equation}
whose integrability conditions are given by the wave equations
$$a_{uu} - a_{vv} = 0 = b_{uu} - b_{vv}.$$

\section{The Weierstrass representation formula in a Lorentzian 3-manifold}\label{weier}


 We will denote by $\K$ either the complex numbers $\C$ or the paracomplex numbers $\L$, and by $\Omega\subset \K$  an open set. Let $(M,g)$ be a Lorentzian 3-manifold and $f:\Omega\subset \K \rightarrow M$ a smooth conformal immersion. We endow $\Omega$ with the induced metric that makes $f$ an isometric immersion. We will say that $f$ is {\it spacelike} if the induced metric on $\Omega$, via $f$, is a Riemannian metric, and that
$f$ is {\it timelike} if the induced metric is a Lorentzian metric.

We observe that in the Lorentzian case, we can endow $\Omega$ with paracomplex isothermic coordinates and, as in the Riemannian case, they are locally described by complex isothermic charts with conformal changes of coordinates. We will denote by $z=u+i\,v$ (resp. $z=u+\tau\,v$) a complex (resp. paracomplex) isothermal coordinate in $\Omega$.

The metric $g$ may be extended for the $\E = f^\ast TM\otimes \K$ as:
\begin{itemize}
\item a (para)complex bilinear form $(.,.):\E\times\E\rightarrow\K$;
\item a (para)Hermitiana metric $\langle\langle.,.\rangle\rangle:\E\times\E\rightarrow\K$;
\end{itemize}
and the two extensions are related by: $$\langle\langle V,W\rangle\rangle = (V,\bar{W}).$$

\begin{thm}[Weierstrass Representation (\cite{Liramm})]\label{teo1}
Let $f:\Omega\subset \K \rightarrow M^3$ be a conformal minimal  spacelike  (or timelike) immersion
and $g = (g_{ij})$ be the induced metric. Define the (para)complex tangent vector $\phi \in \Gamma(f^\ast TM\otimes \K)$ by
\begin{equation}
\phi(z) = \frac{\partial f}{\partial z}\bigg{|}_{f(z)}=\sum_i \phi_i\frac{
\partial}{\partial x_i}.\nonumber
\end{equation}
Then  $\phi_j$, $j = 1,2,3$, satisfy the following conditions:
\begin{itemize}
\item[i)]$\langle\langle \phi,\phi\rangle\rangle  \neq 0,$
\item[ii)] $(\phi,\phi) = 0,$
\item[iii)] $\displaystyle\frac{\partial\phi_k}{\partial\bar{z}} + \sum_{i,j=1}^{3} \Gamma_{ij}^k \bar{\phi_i}\phi_j = 0$,
\end{itemize}
where $\{\Gamma_{ij}^k\}$ are the Christoffel symbols of $M$.
Conversely, if $\Omega\subset\K$ is a simply connected domain and $\phi_j:\Omega\to \K$, $j = 1,2,3$, are (para)complex functions satisfying the
conditions above, then the map
\begin{equation}
f:\Omega\to M,\qquad f_j(z) = 2\,\Re\int_{z_0}^{z} \phi_j\, dz,
\end{equation}
is a well-defined conformal spacelike (or timelike) minimal immersion (here $z_0$ is an arbitrary fixed point of $\Omega$ and the integral
is along any curve joining $z_0$\ to $z$).
\end{thm}

The equation $iii)$ in the Theorem~\ref{teo1} is a system of partial differential equations and, in general, it is quite hard to find explicit solutions. However, in certain spaces, like the Lie groups, these equations become a system of partial differential equations with constant coefficients.


\subsection{The case of Lorentzian Lie Groups}\label{lie}

Let $M$ be a 3-dimensional Lie group endowed with a left-invariant Lorentzian metric $g$
and $\{E_1, E_2, E_3\}$ be left-invariant orthonormal frame field, with $E_1, E_2$ spacelike and $E_3$ timelike. For tangent vectors $W=\sum_{i=1}^3 w_i\,E_i$ e $Y=\sum_{i=1}^3 y_i\,E_i$,
 the Lorentzian cross product  $Y \times W$ is given by:
$$
 Y \times W = (y_2\,w_3-w_2\,y_3)\,E_1+(y_3\,w_1-w_3\,y_1)\,E_2+(y_2\,w_1-w_2\,y_1)\,E_3.
$$
 It is easy to check that $Y \times W = - W \times Y$ and, also,
 \begin{equation}\label{eq:5.1}
 \begin{aligned}
 &g(U \times Y, W \times V) = g(U,V)\, g(Y, W) -
 g(U, W)\,g(Y,V),\\
&(U \times Y) \times W = g(Y,W)\, U - g(U, W)\,Y.
 \end{aligned}
 \end{equation}

Let $f:\Omega\subset \K \rightarrow M$ be a conformal minimal spacelike  (or timelike) immersion, where
$\Omega \subset \K$ is an open set.
Fixed a isothermal parameter $z\in\Omega$, we can write the (para)complex
tangent field $\phi = \displaystyle\frac{\partial f}{\partial z}$ along $f$ both in terms of local coordinates $x_1,x_2,x_3$  in $M$ and also using the left-invariant frame field. Hence, one has
\begin{equation}
\phi = \sum_{a=1}^{3} \phi_a\frac{\partial}{\partial x_a} = \sum_{a=1}^{3} \psi_a E_a,\nonumber
\end{equation}
where the functions $\phi_a$ and $\psi_b$, with $a,b=1,2,3$, are related by
\begin{equation}\label{eq:3.1}
\phi_a = \sum_{b=1}^{3} A_{ab}\psi_b,\quad a = 1,2,3,
\end{equation}
where $A: \Omega \rightarrow GL(3,\r)$. In terms of the components $\psi_a$, $a=1,2,3$, the equation $iii)$ in Theorem~\ref{teo1}
may be written as
$$
\frac{\partial\psi_c}{\partial\bar{z}} + \frac{1}{2}\sum_{a,b=1}^{3} L_{ab}^c \bar{\psi_a}\psi_b = 0,
\quad c = 1,2,3,
$$
where the symbols $L_{ab}^c$ are defined by
$$
\nabla_{E_a} E_b = \sum_{c=1}^{3} \frac{L_{ab}^c}{2} E_c \quad a,b =1,2,3.
$$
Let $C_{ab}^c$ be the structure constants of the Lie algebra of $M$, i.e., $[E_a,E_b] = \sum_{c=1}^{3} C_{ab}^c E_c$.
Therefore, by the Levi-Civita Theorem, we have
\begin{equation}\label{Labc}
L_{ab}^c = (C_{ab}^c - C_{bc}^a\,\varepsilon_a\,\varepsilon_c - C_{ac}^b\,\varepsilon_b\,\varepsilon_c),
\end{equation}
where $\varepsilon_a = \langle E_a, E_a\rangle$, with $a = 1,2,3$.

Consequently, the Theorem~\ref{teo1} may be written in the case of 3-dimensional Lie groups as follows.
\begin{thm}[\cite{Liramm}]\label{teo3.6}
Let $M$ be 3-dimensional Lie group endowed with a left-invariant Lorentzian metric and $\{E_1,E_2,E_3\}$ a left-invariant orthonormal frame field. Let $f: \Omega \rightarrow M$ be a conformal minimal immersion, where $\Omega \subset \K$ is an open set. We denote by $\phi \in \Gamma(f^\ast TM\otimes \K)$ the (para)complex tangent vector
$$
\phi(z) = \frac{\partial f}{\partial z}.
$$
Then, the components $\psi_a$, $a= 1, 2, 3$, of $\phi$ defined by
\begin{equation}
\phi(z) = \sum_{a=1}^{3} \psi_a E_a|_{f(z)},\nonumber
\end{equation}
satisfy the followings conditions:
\begin{enumerate}\label{eq:3.2}
  \item [i)] $|\psi_1|^2 + |\psi_2|^2 - |\psi_3|^2 \neq 0$,
  \item [ii)]${\psi_1}^2 + {\psi_2}^2 - {\psi_3}^2 = 0$,
  \item [iii)]$\displaystyle\frac{\partial\psi_c}{\partial\bar{z}} +
  \sum_{a,b=1}^{3} \frac{L_{ab}^c}{2} \bar{\psi_a}\psi_b = 0$.
\end{enumerate}
Conversely, if $\Omega\subset\K$ is a simply connected domain and $\psi_a:\Omega\to\K$, $a = 1,2,3$, are (para)complex functions satisfying ~the conditions above, then the map $f:\Omega\to M$ which coordinate components are given by:
\begin{equation}\label{minimalfa}
f_a = 2\,\Re \int  \sum_{b=1}^{3} A_{ab}\,\psi_b\, dz, \qquad a = 1,2,3,
\end{equation}
is a well-defined conformal minimal immersion.
\end{thm}

\begin{rem}\label{conservation} We observe explicitly that the equations above does not give the coordinates components of the immersion $f$ just by a direct integration since the functions $A_{ab}$ must be computed along the solutions $f_a$. The formula \eqref{minimalfa} is in fact an integral equation. However, as we will see, for special ambient manifolds this problem can be avoided by {\it ad hoc} arguments.
\end{rem}

\section{The Bj\"{o}rling problem for three-dimensional Lie groups}\label{bjorling}

We denote by $M$ a 3-dimensional Lie group endowed with a left-invariant Lorentzian metric $g$.
Let $\beta: I \rightarrow M$ be a analytic curve in $M$ and $V : I \rightarrow TM$ a unitary  analytic spacelike  (respectively, timelike) vector field along $\beta$, such that $g(\dot{\beta}, V)\equiv 0$. The Bj\"{o}rling problem can be formulate as follows:
\bigskip{}

\noindent {\it  Determine a timelike (respectively, spacelike) minimal
surface $$f: I \times (-\epsilon,\epsilon) = \Omega\subseteq\K \rightarrow M$$ such that
\begin{enumerate}
\item[i.] $f(u,0) = \beta(u)$,
\item[ii] $N(u,0) = V(u)$,
\end{enumerate}
for all $u\in I$, where $N :\Omega \rightarrow TM$ is the Gauss map of the surface.}
\bigskip{}

Before showing that the above problem have a unique solution, we prove the following:
\begin{lem}\label{lem5.2}
Let $\psi_i:\Omega \subseteq \K \rightarrow \K$, $i=1,2$, be two differentiable functions and
$\psi_3^2 = \psi_1^2 + \psi_2^2$. We suppose that $\psi_i$, $i=1,2$, satisfy the two first equations of the third item in Theorem~\ref{teo3.6}. Then $\psi_3$ satisfy the third equation.
\end{lem}
\begin{proof}
Differentiating $\psi_3^2 = \psi_1^2 + \psi_2^2$ with respect to $\bar{z}$ and using that $\psi_i$, $i=1,2$, satisfy
the two first equations of item $iii)$ in Theorem~\ref{teo3.6}, we have that

\begin{eqnarray}
\psi_3\frac{\partial \psi_3}{\partial\bar{z}} &=& \psi_1\frac{\partial \psi_1}{\partial\bar{z}} +
\psi_2\frac{\partial \psi_2}{\partial\bar{z}}
= - \frac{1}{2}\sum_{j,k=1}^{3}[L_{jk}^1\psi_1 + L_{jk}^2\psi_2] \bar{\psi_j}\psi_k.\nonumber
\end{eqnarray}
Therefore, to prove the lemma, it suffices to show that
$$
\sum_{j,k=1}^{3}[L_{jk}^1\psi_1 + L_{jk}^2\psi_2 - L_{jk}^3\psi_3] \bar{\psi_j}\psi_k = 0.
$$
We may write the above sum as follows:
$$
\sum_{i=1}^{3}\{L_{i1}^1{\psi_1}^2 + L_{i2}^2{\psi_2}^2 - L_{i3}^3{\psi_3}^2  +
(L_{i1}^2 + L_{i2}^1)\psi_1\psi_2 + (L_{i3}^1 - L_{i1}^3)\psi_1\psi_3 +
(L_{i3}^2 - L_{i2}^3)\psi_3\psi_2\}\bar{\psi_i}.
$$
Now, using \eqref{Labc}, we have that
$$L_{ik}^k = L_{i1}^2 + L_{i2}^1 = L_{i3}^1 - L_{i1}^3 = L_{i3}^2 - L_{i2}^3 = 0,\qquad
i = 1,2,3.$$
Then
$$
\frac{\partial \psi_3}{\partial\bar{z}} + \frac{1}{2}\sum_{j,k=1}^{3}L_{jk}^3\bar{\psi_j}\psi_k = 0.
$$

\end{proof}

\subsection{The Bj\"{o}rling Problem for timelike surfaces}

We observe that if $f$ is a timelike conformal minimal immersion we have that $\psi_i = \dfrac{\partial f_i}{\partial z}$, $i=1,2,3$, satisfy the condition $iii)$ in Theorem~\ref{teo3.6}, which is equivalent to the hyperbolic system of partial differential equations (see \cite{ev}):

\begin{equation}\label{eqquaselinaer}
\frac{\partial^2 f_i}{\partial u ^2} - \frac{\partial^2 f_i}{\partial v ^2} + B_i\Big(\frac{\partial f_i}{\partial u} ,
\frac{\partial f_i}{\partial v}\Big) = 0,
\end{equation}
where $B_i$, $i=1,2,3$, contain at most first-order derivatives of the functions $f_i$.

So the Bj\"{o}rling problem may be interpreted as a Cauchy problem
involving quasilinear partial differentiable equations~\eqref{eqquaselinaer} with initial data:
$$
f_i(u,0) = \beta_i(u), \quad \quad (V(u) \times \dot{\beta}(u))_i =  \Big(\frac{\partial f}{\partial v}(u,0)\Big)_i, \quad i=1,2,3.
$$

\begin{rem}
Let $\gamma(s) = (u(s),v(s))$ be a characteristic curve in $\Omega$ (see \cite{ev}), then
$$
u'(s)^2 - v'(s)^2 = 0
$$
that is, $\gamma$ ia a straight line $u = \pm v$ in $\Omega$.

It is known that the Cauchy problem may not have a unique solution or it does not have solutions at all if the initial data is along characteristic curves. Moreover, this lines correspond to the lightlike curves of the Bj\"{o}rling problem. Consequently, we consider two cases: when $\beta$ is a timelike curve we will call the corresponding problem the
{\it timelike Bj\"{o}rling problem}, and when $\beta$ is a spacelike curve we will have the {\it spacelike Bj\"{o}rling problem}.
\end{rem}

\begin{thm}[\textbf{Timelike Bj\"{o}rling Problem}]\label{PdeBdot-t}
Let $\beta: I \rightarrow M$ be a analytic timelike curve in $M$ and $V : I \rightarrow TM$ a unitary
analytic spacelike vector field along $\beta$, such that $g(\dot{\beta}, V)\equiv 0$. Then, there exists a unique conformal timelike minimal
surface $$f: I \times (-\epsilon,\epsilon) = \Omega\subseteq\L \rightarrow M$$ such that
\begin{enumerate}
\item[i.] $f(u,0) = \beta(u)$,
\item[ii] $N(u,0) = V(u)$,
\end{enumerate}
for all $u\in I$, where $N :\Omega\rightarrow TM$ is the Gauss map of the surface.
\end{thm}

\proof
Consider the system
\begin{equation}\label{eq:5.4}\left\{\begin{aligned}
\displaystyle\frac{\partial \psi_1}{\partial\bar{z}} + \sum_{j,k=1}^{3}L_{jk}^1\bar{\psi_j}\psi_k &= 0,\\
\displaystyle\frac{\partial \psi_2}{\partial\bar{z}} + \sum_{j,k=1}^{3}L_{jk}^2\bar{\psi_j}\psi_k &= 0,
\end{aligned}\right.
\end{equation}
where $\psi_i : \Omega \rightarrow \L$ and $\psi_3^2 = \psi_1^2 + \psi_2^2$.

As  $\beta$ is a timelike curve, then it's not a characteristic curve and, so, this system is of Cauchy-Kovalevskaya
type (\cite{P}).  Therefore, fixing the initial data, it admits a unique solution (locally).
Hence, we must find the initial data so that the minimal surface has the required properties. We observe that, if $f$  is a solution of the Bj\"{o}rling problem, we have that
\begin{equation}
\frac{\partial f}{\partial u}(u,0) = \dot{\beta}(u)\quad\mbox{e}\quad\frac{\partial f}{\partial v}(u,0) = V(u) \times \dot{\beta}(u).
\end{equation}
Then
\begin{equation}\label{phiu0}
\phi(u,0) = \frac{1}{2}\big{(}\frac{\partial f}{\partial u} + \tau\frac{\partial f}{\partial v}\big{)}(u,0) =
\frac{1}{2}\big{(} \dot{\beta}(u) + \tau\, V(u)\times\dot{\beta}(u)\big{)}.
\end{equation}
So the initial condition for the system is given by
\begin{equation}
\psi(u,0) = A^{-1}(\beta(u))\,\phi(u,0),\nonumber
\end{equation}
where $A$ is given by \eqref{eq:3.1}.

We note that Lemma~\ref{lem5.2} implies that the functions $\psi_i$ satisfy the equations ii) and iii) in
the Weierstrass representation Theorem~\ref{teo3.6}.
Furthermore, from \eqref{phiu0}, it follows that
$$\langle\langle \phi(u,0),\phi(u,0)\rangle\rangle  = \frac{1}{4}[g(\dot{\beta},\dot{\beta}) -
g(V\times\dot{\beta},V\times\dot{\beta})] < 0,$$
because, from \eqref{eq:5.1}, it results that $g(V\times\dot{\beta},V\times\dot{\beta}) = -g(\dot{\beta},\dot{\beta})$ and
$\beta$ is a timelike curve.
Shrinking $\Omega$ if necessary, we can assume that $$\langle\langle \phi(u,v),\phi(u,v)\rangle\rangle < 0, \quad (u,v)\in\Omega.$$
Since $A$ is the Jacobian matrix of a left-invariant translation in $M$, it results that
$$|\psi_1(u,v)|^2 + |\psi_2(u,v)|^2 - |\psi_3(u,v)|^2= \langle\langle\phi(u,v),\phi(u,v)\rangle\rangle < 0,\quad (u,v)\in\Omega.$$

Therefore the functions $\psi_i$, $i=1,2,3$, satisfy the conditions of Theorem~\ref{teo3.6}. So, from the Cauchy-Kovalevskaya Theorem, there exists a unique conformal timelike minimal
immersion, which is a local solution of the timelike Bj\"{o}rling problem. Observe that the initial condition forces the choice of one of the determinations of $\psi_3^2 = \psi_1^2 + \psi_2^2$.

So we have the local solution. We may be consider that $I$ is compact. Since the solution is locally unique
if $\beta(I)$ is contained in a coordinate neighborhood for $\epsilon>0$ sufficiently small, we have that solution is unique
because $I$ is compact. If it is not, as $I$ is compact, we can cover it with a finite number of inverse images
of neighborhoods coordinates through $\beta$ and using again the local uniqueness of the problem
we obtain the global solution.
\endproof

We can prove that the spacelike Bj\"{o}rling problem has a unique solution analogously to the timelike case. In this case the initial data is
$$\psi(u,0) = A^{-1}(\beta(u))\,\phi(u,0),$$
where
$$
\phi(u,0) = \frac{1}{2}(\dot{\beta}(u) - \tau\, V(u)\times \dot{\beta}(u)).
$$

\begin{thm}[\textbf{Spacelike Bj\"{o}rling problem}]\label{PdeBdot-e}
Let $\beta: I \rightarrow M$ be a analytic spacelike curve in $M$ and $V : I \rightarrow TM$ a unitary
analytic spacelike vector field along $\beta$, such that $g(\dot{\beta}, V)\equiv 0$. Then there exists a unique conformal timelike minimal
surface $$f: I \times (-\epsilon,\epsilon) = \Omega\subseteq\L \rightarrow M$$ such that:
\begin{enumerate}
\item[i.] $f(u,0) = \beta(u)$,
\item[ii] $N(u,0) = V(u)$,
\end{enumerate}
for all $u\in I$, where $N :\Omega\rightarrow TM$ is the Gauss map of the surface.
\end{thm}

\subsection{The Bj\"{o}rling Problem for spacelike surfaces}

The Bj\"{o}rling problem for spacelike surfaces has a unique solution and the proof is analogously to the case of the timelike surfaces. In this case the system of condition $iii)$ in Theorem~\ref{teo3.6} is equivalent to the elliptic system of partial differential equations (see \cite{ev}):

\begin{equation}\label{eqquaselinaere}
\frac{\partial^2 f_i}{\partial u ^2} + \frac{\partial^2 f_i}{\partial v ^2} + B_i\Big(\frac{\partial f_i}{\partial u} ,
\frac{\partial f_i}{\partial v}\Big) = 0,
\end{equation}
where $B_i$, $i=1,2,3$, contain at most first-order derivatives of the functions $f_i$.

So the Bj\"{o}rling problem may be interpreted as a Cauchy problem
involving quasilinear partial differentiable equations~\eqref{eqquaselinaere} with initial data:
$$
f_i(u,0) = \beta_i(u), \quad \quad (V(u) \times \dot{\beta}(u))_i =  -\Big(\frac{\partial f}{\partial v}(u,0)\Big)_i, \quad i=1,2,3.
$$
Therefore we can use again the Cauchy-Kovalevskaya Theorem (see \cite{ev}) to show that the problem has a unique solution with a initial data given by
$$\psi(u,0) = A^{-1}(\beta(u))\,\phi(u,0),$$
where
$$
\phi(u,0) = \frac{1}{2}(\dot{\beta}(u) + i\, V(u)\times \dot{\beta}(u)).
$$

\begin{thm}[\textbf{Bj\"{o}rling problem}]\label{PdeBforss}
Let $\beta: I \rightarrow M$ be a analytic spacelike curve in $M$ and $V : I \rightarrow TM$ a unitary
analytic timelike vector field along $\beta$, such that $g(\dot{\beta}, V)\equiv 0$. Then there exists a unique conformal spacelike minimal surface $$f: I \times (-\epsilon,\epsilon) = \Omega\subseteq\C \rightarrow M$$ such that:
\begin{enumerate}
\item[i.] $f(u,0) = \beta(u)$,
\item[ii] $N(u,0) = V(u)$,
\end{enumerate}
for all $u\in I$, where $N :\Omega\rightarrow TM$ is the Gauss map of the surface.
\end{thm}

Now we will construct some examples of minimal surfaces in the Heisenberg group $\h_3$, in the de Sitter Space $\mathbb{S}_1^3$ and in the space $\mathbb{H}^2\times\mathbb{R}$.

\section{The Lorentzian Heisenberg group $\h_3$}\label{lorentzianH}
We consider the Heisenberg group
$$
\h_3 = \Bigg{\{}\left[\begin{array}{cccl}
  1 & x & z + \frac{1}{2}xy \\
  0 & 1 &     y\\
  0 & 0 &     1
\end{array}\right]\;:\; x,y,z \in \r \Bigg{\}},
$$
equipped with the left-invariant Lorentzian metric given by
\begin{equation}
g = dx^2 + dy^2 - \Big{(}\frac{1}{2}y\,dx - \frac{1}{2}x\,dy + dz\Big{)}^2.\nonumber
\end{equation}
With respect to $g$, the left-invariant frame field given by
$$
E_1= \frac{\partial}{\partial x} - \frac{y}{2}\frac{\partial}{\partial z},\quad
E_2= \frac{\partial}{\partial y} + \frac{x}{2}\frac{\partial}{\partial z},\quad
E_3=  \frac{\partial}{\partial z},
$$
is orthonormal, $\{E_1,E_2\}$ are spacelike and $E_3$ is timelike. Also, the matrix $A$ is
\begin{equation}
A = \left[\begin{array}{cccl}
      1        &     0       & 0 \\
      0        &     1       & 0 \\
  -\frac{y}{2} & \frac{x}{2} & 1
\end{array}\right],\nonumber
\end{equation}
and the nonzero $L_{ij}^k$ are $L_{12}^3 = L_{13}^2  = L_{31}^2 = \frac{1}{2}$ and
$L_{21}^3 = L_{32}^1  = L_{23}^1 = - \frac{1}{2}$.
Consequently the system~\eqref{eq:5.4} becomes
\begin{equation}\label{eq:5.6} \left\{\begin{aligned}
\frac{\partial\psi_1}{\partial\bar{z}} &- \Re(\bar{\psi_3}\psi_2) = 0,\\
\frac{\partial\psi_2}{\partial\bar{z}} &+ \Re(\bar{\psi_3}\psi_1) = 0.
\end{aligned}\right.
\end{equation}
Then the coordinates components of the minimal immersion $f$ are given by
\begin{equation}\label{eq:5.5} \left\{\begin{aligned}
f_1 &= 2\,\Re\int \psi_1\, dz,\\
f_2 &= 2\,\Re\int \psi_2 \,dz,\\
f_3 &= 2\,\Re\int \Big(\frac{f_1}{2}\psi_2 -\frac{f_2}{2}\psi_1 + \psi_3\Big) \,dz.
\end{aligned}\right.
\end{equation}
So, knowing the $\psi_i,\, i=1,2,3$, that are solutions of a constant coefficients PDE, we can compute $f_1,f_2$ by integration and, substituting in the third equation of \eqref{eq:5.5}, we can compute $f_3$ by direct integration (see Remark~\ref{conservation}).

\begin{ex}[The timelike vertical plane $y = c$, timelike case]\label{Pvertical-t-t}
First of all, we consider the curve  $$\beta(u) = (\cosh u, c, -\frac{c}{2}\cosh u + \sinh u), \qquad u\in\r,\quad c\in\r,$$
and the unit vector field $V(u) = E_2(\beta(u))$.
Since
$$
\dot{\beta}(u) = \sinh u\,E_1 + \cosh u\,E_3,
$$
then $g(\dot{\beta},\dot{\beta}) = -1$. Moreover, $g(V,V) = 1$ and $g(\dot{\beta}, V) = 0$.
Thus we have a {\em timelike Bj\"{o}rling problem}.
As
$$E_2\times E_1 = E_3\quad \mbox{and}\quad E_2\times E_3 = E_1,$$
we have that $V(u)\times\dot{\beta}(u) = \cosh u\,E_1+\sinh u\,E_3 $. So
\begin{equation}
\phi(u,0) = \frac{1}{2}[(\sinh u + \tau\cosh u)\,E_1 + (\cosh u + \tau\sinh u)\,E_3].\nonumber
\end{equation}
Consequently,
\begin{equation}\label{eq:5.7}
\psi(u,0) = A^{-1}(\beta(u))\phi(u,0)
          = \frac{1}{2}(\sinh u + \tau\cosh u, 0, \cosh u + \tau\sinh u).
\end{equation}
Therefore the solution of system \eqref{eq:5.6}, which satisfy the initial condition \eqref{eq:5.7}, is
\begin{equation}\left\{\begin{aligned}\label{eq:5.8}
\psi_1(u,v) &= \frac{e^v}{2}(\sinh u + \tau\cosh u),\\
\psi_2(u,v) &= 0,\\
\psi_3(u,v) &= \frac{e^v}{2}(\cosh u + \tau\sinh u).
\end{aligned}\right.
\end{equation}
Furthermore, \eqref{eq:5.8} satisfy the conditions of Theorem~\ref{PdeBdot-t}. Then, we obtain the conformal minimal timelike immersion
$$
f(u,v) = \big{(}e^v\,\cosh u ,c, e^v\,(-\frac{c}{2}\cosh u + \sinh u)\big{)}$$
that is a timelike vertical plane $y=c$, and it represents the solution of the Bj\"{o}rling problem for the given pair $(\beta,V)$.
\end{ex}

\begin{ex}[Helicoids]
Consider $\beta(u) = (\rho(u), 0, b)$, $b\in\r$  and
$$V(u) = \frac{\rho^2(u) -2c}{2\rho'(u)}\,E_2 - \frac{\rho(u)}{\rho'(u)}\,E_3,\quad u\in (a,d)\subset\r, \quad c\in\r,$$
where $\rho$ is a real function satisfying
\begin{equation}
\sqrt{(\rho')^2 + \rho^2} = \frac{\rho^2}{2} - c.\nonumber
\end{equation}
As $\dot{\beta}(u) = \rho'(u)\,E_1$, then $g(\dot{\beta},\dot{\beta}) = \rho'^2$.
Moreover $$g(V,V) = 1,\qquad g(\dot{\beta},V) = 0.$$ Thus we have a {\em spacelike
Bj\"{o}rling problem}.
Since $E_2\times E_1 = E_3$ and $E_3\times E_1 =  E_2$, we obtain
\begin{equation}
V(u)\times\dot{\beta}(u) = \frac{\rho^2(u) - 2c}{2}\,E_3 - \rho(u)\, E_2.
\end{equation}
Then
$$
\phi(u,0) = \Big(\frac{\rho'(u)}{2}, \frac{\tau\,\rho(u)}{2},
\frac{c\,\tau}{2}\Big)$$
and, so, it follows that
\begin{equation}\label{eq:5.11}
\psi(u,0) = \Big(\frac{\rho'(u)}{2}, \frac{\tau\,\rho(u)}{2}, - \frac{\tau(\rho^2(u) - 2c)}{4}\Big).
\end{equation}
Therefore the solution of \eqref{eq:5.6}, which satisfy the initial condition \eqref{eq:5.11}, is
\begin{equation}\left\{\begin{aligned}\label{eq:5.12}
\psi_1(u,v) &= \frac{1}{2}(\rho'(u)\,\cos v - \tau\rho(u)\,\sin v),\\
\psi_2(u,v) &= \frac{1}{2}(\rho'(u)\,\sin v + \tau\rho(u)\,\cos v),\\
\psi_3(u,v) &= -\frac{\tau}{4}(\rho^2(u) - 2c).
\end{aligned}\right.
\end{equation}
Furthermore, \eqref{eq:5.12} satisfy the conditions of Theorem~\ref{PdeBdot-e}. Then, integrating we obtain the solution to the Bj\"{o}rling problem:
\begin{equation}
f(u,v) = (\rho(u)\,\cos v,\rho(u)\,\sin v,c\,v + b),\nonumber
\end{equation}
which represents a timelike helicoid if $c\neq 0$, and the horizontal plane $z=b$, if $c=0$.
\end{ex}

\begin{ex}[The saddle-type surface]
Consider $\beta(u) = (4\,c\,u, -4\,Q(0), -8\,c\,u\,Q(0))$ and
$$V(u) = -\dfrac{4\,c\,Q(0)}{Q'(0)}\,E_1 + \dfrac{c}{Q'(0)}\,E_3, \quad u\in (a,b)\subset\r, \quad c\in\r$$
where $Q(v)$ is a real differential function with $Q'(v)\neq 0$, for all $v\in\r$, which satisfy
$$4\,c\,Q(v) = \sqrt{Q'(v)^2 + c^2}.$$
As $$\dot{\beta}(u) = 4\,c\,E_1 - 16\,c\,Q(0)\,E_3,$$ then $g(\dot{\beta}(u),\dot{\beta}(u)) = -16\,Q'(0)^2$. Moreover $g(V(u),V(u)) = 1$ and $g(\dot{\beta}(u),V(u)) = 0$. Thus we have a {\em timelike Bj\"{o}rling problem}. Since
$$V(u)\times\dot{\beta}(u) = -4\,Q'(0)\,E_2,$$
then
$$\phi(u,0) = \big{(}2\,c,-2\,Q'(0)\tau, -4\,c\,Q(0) -4\,a\,u\,Q'(0)\tau\big{)}$$
and, so, it follows that
\begin{equation}\label{eq:5.13}
\psi(u,0) = \big{(}2\,c, -2\,Q'(0)\tau, -8\,c\,Q(0)\big{)}.
\end{equation}
Therefore the solution of \eqref{eq:5.6}, which satisfy the initial condition \eqref{eq:5.13} is
\begin{equation}\left\{\begin{aligned}\label{eq:5.14}
\psi_1(u,v) &= 2\,c,\\
\psi_2(u,v) &= -2\,\tau\,Q'(v),\\
\psi_3(u,v) &= -8\,c\,Q(v).
\end{aligned}\right.
\end{equation}
Furthermore, \eqref{eq:5.14} satisfy the conditions of the Theorem~\ref{PdeBdot-t}. Then, integrating we obtain the
$$f(u,v) = (4\,c\,u,-4\,Q(v),-8\,c\,u\,Q(v)),$$
which the image lies on the graph of the function $z = \dfrac{1}{2}x\,y$.
\end{ex}

\section{The de Sitter Space $\mathbb{S}_1^3$}\label{sitter}

The de Sitter space $\mathbb{S}_1^3$ might be modeled as the halfspace
$$
\r_{+}^3 = \{(x_1,x_2,x_3) \in \r^3\;: \;x_3 > 0\}$$
endowed with the left-invariant Lorentzian metric given by
$$
g = \frac{1}{x_3^2}(dx_1^2 + dx_2^2 - dx_3^2).
$$
An orthonormal basis of left-invariant vector fields is given by
$$
E_1 = x_3\frac{\partial}{\partial x_1},\qquad
E_2 = x_3\frac{\partial}{\partial x_2},\qquad
E_3 = x_3\frac{\partial}{\partial x_3},
$$
where $\{E_1,E_2\}$ are spacelike and $E_3$ timelike. Then
\begin{equation}
A = \left[\begin{array}{cccl}
 x_3 &  0  &  0 \\
  0  & x_3 &  0 \\
  0  &  0  & x_3
\end{array}\right].\nonumber
\end{equation}
The only $L_{ij}^k$ nonzero are $L_{13}^1 = L_{23}^2  = L_{11}^3 = L_{22}^3 = -1$. So \eqref{eq:5.4} becomes
\begin{equation}\label{eq:5.15} \left\{\begin{aligned}
\frac{\partial\psi_1}{\partial\bar{z}} &- \bar{\psi_1}\psi_3 = 0, \\
\frac{\partial\psi_2}{\partial\bar{z}} &- \bar{\psi_2}\psi_3 = 0.
\end{aligned}\right.
\end{equation}
Therefore the conformal minimal immersion $f$ is given by
\begin{equation}\label{eq:5.16} \left\{\begin{aligned}
f_1 &= 2\,\Re\int f_3\,\psi_1\, dz,\\
f_2 &= 2\,\Re\int f_3\,\psi_2\, dz,\\
f_3 &= \exp\Big({2\,\Re\int\psi_3\, dz}\Big).
\end{aligned}\right.
\end{equation}
Again, knowing the $\psi_i,\, i=1,2,3$, we can compute $f_3$ by direct integration and, substituting in the first two equations, we can also compute $f_1,f_2$ by direct integration (see, again, Remark~\ref{conservation}).

\begin{ex}[Timelike vertical plane $y=c$, the  spacelike case]
Consider $$\beta(u) = (\sinh u, c, \cosh u), \quad u\in\r $$ and $V(u) = E_2$.
We have that $$\dot{\beta}(u) = E_1 + \frac{\sinh u}{\cosh u}\,E_3,\qquad
g(\dot{\beta}(u),\dot{\beta}(u)) = \frac{1}{\cosh^2(u)}>0.$$ Furthermore, $g(V,V) = 1$ and $g(\dot{\beta},V) = 0$.
Thus the pair $(\beta,V)$ produces a {\em spacelike Bj\"{o}rling problem}.
Since $$V(u)\times\dot{\beta}(u) = E_3 + \frac{\sinh u}{\cosh u}\,E_1 = (\sinh u, 0, \cosh u),$$
we obtain that
$$
\phi(u,0) = \frac{1}{2}(\dot{\beta}(u) - \tau\, V(u)\times\dot{\beta}(u))
          = \frac{1}{2}( \cosh u - \tau\sinh u, 0, \sinh u - \tau\cosh u).
$$
Therefore
\begin{equation}
\psi(u,0) = A^{-1}(\beta(u))\phi(u,0)
          = \frac{1}{2}\bigg{(}\frac{\cosh u - \tau\sinh u}{\cosh u}, 0, \frac{\sinh u
          - \tau\cosh u}{\cosh u}\bigg{)}.\nonumber
\end{equation}
As $\psi(u,0)$ is a solution of \eqref{eq:5.15}, the uniqueness implies that $\psi(u,v) = \psi(u,0)$.
Therefore the conformal minimal timelike immersion, which contains $\beta(u)$, is given by
\begin{equation}
f(u,v) = (e^{-v}\,\sinh u, c, e^{-v}\,\cosh u).\nonumber
\end{equation}
\end{ex}

\begin{ex}
Consider $$\beta(u) = \bigg{(}\dfrac{1}{\sqrt{2}}\sinh u, \dfrac{1}{\sqrt{2}}\sinh u, \cosh u\bigg{)},\qquad V(u) = -\dfrac{E_1}{\sqrt{2}} + \dfrac{E_2}{\sqrt{2}}.$$
As
$$\dot{\beta}(u) = \dfrac{E_1}{\sqrt{2}} + \dfrac{E_2}{\sqrt{2}} + \frac{\sinh u}{\cosh u}\,E_3$$
and
$$g(\dot{\beta}(u),\dot{\beta}(u)) = \frac{1}{\cosh^2(u)},$$
then we have a {\em spacelike Bj\"{o}rling problem}. Also,
$$g(V(u),V(u)) = 1,\qquad g(\dot{\beta}(u),V(u)) = 0$$
and $$V(u)\times\dot{\beta}(u) =  \bigg{(}\dfrac{1}{\sqrt{2}}\sinh u, \dfrac{1}{\sqrt{2}}\sinh u, \cosh u\bigg{)}.$$
Therefore
$$
\phi(u,0) = \frac{1}{2}\bigg{(}\frac{1}{\sqrt{2}}(\cosh u - \tau\sinh u),\frac{1}{\sqrt{2}}(\cosh u - \tau\sinh u) , \sinh u - \tau\cosh u\bigg{)}$$
and
\begin{equation}
\psi(u,0) = A^{-1}(\beta(u))\phi(u,0)
          = \frac{1}{2}\bigg{(}\frac{\cosh u - \tau\sinh u}{\sqrt{2}\cosh u}, \frac{\cosh u - \tau\sinh u}{\sqrt{2}\cosh u}, \frac{\sinh u - \tau\cosh u}{\cosh u}\bigg{)}.\nonumber
\end{equation}
Since $\psi(u,0)$ is solution of \eqref{eq:5.15}, the uniqueness implies that $\psi(u,v) = \psi(u,0)$.
Integrating \eqref{eq:5.16}, it results that
$$
f(u,v) = e^{-v}\Big(\dfrac{\sinh u}{\sqrt{2}}, \dfrac{\sinh u}{\sqrt{2}}, \cosh u\Big).
$$
\end{ex}

\section{The space $\mathbb{H}^2\times\mathbb{R}$}\label{H2xR}

Let hyperbolic space $\mathbb{H}^2$ be modeled as halfspace
$$
\r_{+}^2 = \{(x_1,x_2) \in \r^2\;: \;x_2 > 0\}$$
endowed with the left-invariant metric given by
$$
g_{\mathbb{H}} = \frac{1}{x_2^2}(dx_1^2 + dx_2^2).
$$
The space $\mathbb{H}^2\times\mathbb{R}$ is a Lie group and with the product structure and endowed with the left-invariant metric given by
$$
g = \frac{1}{x_2^2}(dx_1^2 + dx_2^2) - dx_3^2.
$$
An orthonormal basis of left-invariant vector fields is given by
$$
E_1 = x_2\frac{\partial}{\partial x_1},\qquad
E_2 = x_2\frac{\partial}{\partial x_2},\qquad
E_3 = \frac{\partial}{\partial x_3},
$$
where $\{E_1,E_2\}$ are spacelike and $E_3$ timelike. Then
\begin{equation}
A = \left[\begin{array}{cccl}
 x_2 &  0  &  0 \\
  0  & x_2 &  0 \\
  0  &  0  & 1
\end{array}\right].\nonumber
\end{equation}
The only $L_{ij}^k$ nonzero are $L_{12}^1 = -2$ and $L_{11}^2 = 2$. So \eqref{eq:5.4} becomes
\begin{equation}\left\{\begin{aligned}\label{epshipReq}
\frac{\partial\psi_1}{\partial\bar{z}} &- \psi_1\bar{\psi_2} = 0,\\
\frac{\partial\psi_2}{\partial\bar{z}} &+ \bar{\psi_1}\psi_1 = 0.\\
\end{aligned}\right.
\end{equation}
Therefore the conformal minimal immersion $f$ is given by
\begin{equation}\left\{\begin{aligned}
f_1(u,v) &= 2\,\Re\int f_2\,\psi_1\, dz,\\
f_2(u,v) &= \exp\Big(2\,\Re\int\psi_2\, dz\Big),\\
f_3(u,v) &= 2\,\Re\int \psi_3\, dz.
\end{aligned}\right.\nonumber
\end{equation}
Again, knowing the $\psi_i,\, i=1,2,3$, we can compute $f_2$ and $f_3$ by direct integration and, substituting in the first equation, we can also compute $f_1$ by direct integration (see, again, Remark~\ref{conservation})

We observe that if $\psi_2$ is a $\mathbb{L}$-differentiable (or holomorphic) function the it follows from \eqref{epshipReq} that $\psi_1\bar{\psi_1} = 0$ and $\psi_1\dfrac{\partial\psi_1}{\partial\bar{z}} = 0$. We may for example $\psi_1 =0$, which corresponds to planes $x_1 = {\text cte}$.

\begin{ex}[Spacelike horizontal plane $z=c$]
Consider $$\beta(u) = (\cos u, \sin u, c), \quad u\in(0,\pi)$$ and $V(u) = E_3$.
We have that $$\dot{\beta}(u) = -E_1 + \frac{\cos u}{\sin u}\,E_2,\qquad
g(\dot{\beta}(u),\dot{\beta}(u)) = \frac{1}{\sin^2(u)}>0.$$ Furthermore, $g(V,V) = -1$ and $g(\dot{\beta},V) = 0$.
Thus the pair $(\beta,V)$ produces a  Bj\"{o}rling problem for spacelike surfaces.
As
 $$V(u)\times\dot{\beta}(u) = -E_2 - \frac{\cos u}{\sin u}\,E_1 = (-\cos u,-\sin u, 0),$$
we obtain that
$$
\phi(u,0) = \frac{1}{2}(\dot{\beta}(u) + i\, V(u)\times\dot{\beta}(u))
          = \frac{1}{2}( -\sin u - i\cos u, \cos u - i\sin u, 0).
$$
Therefore
\begin{equation}
\psi(u,0) = A^{-1}(\beta(u))\phi(u,0)
          = \frac{1}{2}\bigg{(}-\frac{\sin u + i\cos u}{\sin u}, \frac{\cos u - i\sin u}{\sin u}, 0\bigg{)}.\nonumber
\end{equation}
Since $\psi(u,0)$ is a solution of \eqref{epshipReq}, the uniqueness implies that $\psi(u,v) = \psi(u,0)$.
Therefore the conformal minimal spacelike immersion, which contains $\beta(u)$, is given by
\begin{equation}
f(u,v) = (e^{v}\,\cos u, e^{v}\,\sin u, c).\nonumber
\end{equation}
\end{ex}

\end{document}